\documentclass[11pt]{amsart}
\usepackage{amsthm}
\usepackage{mathtools}
\usepackage{enumitem}
\usepackage{graphicx, amsmath,latexsym, amsfonts, amssymb, amstext}
\usepackage{mathrsfs}
\usepackage{color}
\usepackage[utf8]{inputenc}
\usepackage[english]{babel}
\usepackage{textcomp}
\usepackage[small,bf]{caption}
\usepackage{amssymb, amsmath, amsfonts, amsthm}
\usepackage{epsfig,amsfonts,amsbsy,amssymb,amsthm}
\usepackage[numbers]{natbib}
\allowdisplaybreaks

\newtheorem{theorem}{Theorem}
\newtheorem{lemma}[theorem]{Lemma}
\newtheorem{corollary}[theorem]{Corollary}

\theoremstyle{definition}

\newtheorem{example}[theorem]{Example}

\numberwithin{theorem}{section}

\newcommand{\ints}{\mathbb{Z}}

\newcommand{\m}{\mathfrak m}

\def\to{\longrightarrow}
\DeclareMathOperator{\h}{H}

\DeclareMathOperator{\depth}{depth}% % % % % % \pagenumbering{roman}

\title[Bounding Hilbert coefficients of parameter ideals]
{Bounding Hilbert coefficients of parameter ideals}

\date{\today}


\author[Saikia] {Anupam Saikia}
\address{Department of Mathematics, Indian Institute of Technology Guwahati, Assam 781039, India}
\email{a.saikia@iitg.ac.in}
\author[Saloni] {Kumari Saloni}
\address{Department of Mathematics, Indian Institute of Technology Guwahati, Assam 781039, India}
\email{saloni.kumari@iitg.ac.in, sin.saloni@gmail.com}

\begin{document}
\sloppy
 \begin{abstract}
 Let $(R,\m)$ be a Noetherian local ring of dimension $d>0$ and $\depth R\geq d-1$. Let $Q$ be a parameter ideal of $R$. In this paper, we derive uniform
 lower and upper bounds for the Hilbert coefficient $e_i(Q)$ under certain assumptions on the depth of associated graded ring $G(Q)$. For $2\leq i\leq d $, we show that
 (1) $e_i(Q)\leq 0$ provided $\depth G(Q)\geq d-2$ and (2) $e_i(Q)\geq -\lambda_R(\h_{\m}^{d-1}(R))$ provided $\depth G(Q)\geq d-1$. It is proved that
 $e_3(Q)\leq 0$.  Further, we obtain a necessary condition for the vanishing of the last coefficient $e_d(Q)$. As a consequence, we characterize the vanishing of $e_2(Q)$. Our results
 generalize \cite[Theorem 3.2]{goto-ozeki} and \cite[Corollary 4.5]{Lori}.
 \end{abstract}
 \maketitle
\section{Introduction}
Let $(R,\m)$ be a Noetherian local ring of dimension $d>0$ and $Q$ an $\m$-primary ideal of $R$. Let $\lambda_R(M)$ denote the length of an $R$-module $M.$
The Hilbert-Samuel function of $Q$ is defined as $H(Q,n)=\lambda_R(R/Q^n)$ for $n\in\ints$. It is well known that $H(Q,n)$  coincides with a polynomial
$P(Q,n)$ of degree $d$ for all $n\gg 0.$ The polynomial $P(Q,x)$
is called the Hilbert-Samuel polynomial of $Q$. We may write
\begin{equation*}
P(Q,x)=e_0(Q)\binom{x+d-1}{d}-e_1(Q)\binom{x+d-2}{d-1}+\ldots+(-1)^de_d(Q)
% % \label{Hilbert-S-poly}
\end{equation*}
for unique integers $e_i(Q)$ known as the Hilbert coefficients of $Q$.

 In the classical case of a Cohen-Macaulay
local ring, relations among various Hilbert coefficients and bounds for them have been explored by several authors. Northcott's inequality \cite{northcott}
$e_1(Q)\geq e_0(Q)-\lambda_R(R/Q)$ is one of the first results in this direction. It was improved by M. E. Rossi in \cite{rossi-reduction}.
% % Huneke \cite{huneke-relations} and Ooishi \cite{ooishi} studied the ideals $Q$ for which the equality $e_1(Q)=e_0(Q)-\lambda_R(R/Q)$ holds.
Several bounds on $e_1(Q)$ in terms of $e_0(Q)$ exist in literature, see \cite{elias}, \cite{elias2}, \cite{hanumanthu-huneke}, \cite{RV} and \cite{RV2}.
For example, Rossi and Valla in \cite[Proposition 2.10]{RV} proved that $e_1(Q)\leq \binom{e_0(Q)-k+1}{2}$ where
 $Q\subseteq \m^k.$
Such bounds are useful for examining the finiteness of Hilbert functions of ideals with fixed multiplicity, see \cite{srinivas-trivedi}.
% % When $R$ is Cohen-Macaulay and $I$ is an $\m$-primary ideal,
With $R$ Cohen-Macaulay, it is known that $e_1(Q)$ and $e_2(Q)$ are non-negative, due to Northcott \cite{northcott} and Narita \cite{narita} respectively.
However the higher coefficients are not necessarily non-negative. Marley \cite[Example 2.3]{Marley} gave an example of a Cohen-Macaulay local ring and an
$\m$-primary ideal $Q$ with $e_3(Q)<0$.
% % Itoh \cite{itoh} showed that $e_3(Q)\geq 0$ if $Q$ is a normal ideal.
% % For a parameter ideal $Q$ in a Cohen-Macaulay local ring, $e_i(Q)=0$ for $i=1,\ldots,d.$ However t
The higher coefficients are not yet explored much except when the associated graded ring have high depth. By depth of a standard graded ring with a
unique graded maximal ideal, we mean the grade of the unique maximal ideal. Let
$G(Q)=\mathop{\oplus}\limits_{n\geq 0} Q^n/Q^{n+1}$ denote the associated graded ring of $Q$. Marley \cite[Corollary 2.9]{Marley} showed
that in a Cohen-Macaulay local ring, $e_i(Q)\geq 0$ for $0\leq i\leq d$ provided depth of $G(Q)$ is at least $d-1$.

The case when $R$ is not Cohen-Macaulay is quite different. Let $Q$ be a parameter ideal of $R$ hereafter.
In this paper, we find uniform lower and upper bounds for the coefficients $e_i(Q)$ for $2\leq i\leq d$. In \cite[Theorem 3.5]{msv},
  Mandal, Singh and Verma  proved that $e_1(Q)\leq 0$.
 Mccune in \cite[Theorem 3.5]{Lori} showed that if $\depth R\geq d-1$ then $e_2(Q)\leq 0$. In the same paper, she proved that $e_i(Q)\leq 0$  for $2\leq i\leq d$
 provided $\depth G(Q)\geq d-1$.
% %  In addition, if $\depth G(Q)\geq d-1$ then she proved that $e_i(Q)\leq 0$  for $2\leq i\leq d$ \cite[Corollary 4.5]{Lori}.
We improve upon Mccune's result by relaxing the hypothesis to $\depth G(Q)\geq d-2$.
If $\depth G(Q)\geq d-1$, we provide a uniform lower bound for $e_i(Q)$ which is independent of $Q$.
 Let
$H_\m^i(*)$ denote the $i$-th local cohomology functor with support in the maximal ideal $\m$.
Goto and Ozeki \cite[Theorem 3.2]{goto-ozeki}
showed that in two dimensional local ring with $\depth R\geq 1$, $-\lambda_R(\h_{\m}^1(R))\leq e_2(Q)\leq 0$. We extend their result to rings of  dimension
$d$ and $\depth$ at least $d-1$.  Our proofs are essentially inspired from the methods developed in \cite{goto-ozeki}.

Let $\mathcal{R}=R(I)=\mathop\oplus\limits_{n=0}^\infty I^nt^n\subseteq R[t]$ and $\mathcal{R}^*=R^*(I)=\mathop\oplus\limits_{n\in\ints}I^nt^n\subseteq R[t,t^{-1}]$
denote the Rees algebra and the extended Rees algebra respectively of an ideal $I$.
 We put $\mathcal{M}=\m \mathcal{R}+\mathcal{R_{+}}$ where $\mathcal{R_{+}}=\mathop\oplus\limits_{n>0}\mathcal{R}_n$
is the irrelevant ideal of the Rees algebra $\mathcal{R}$. Let $[T]_n$ denote the $n$-th graded piece of a graded $\mathcal{R}$-module $T$.
This paper is organized as follows.

In Section \ref{section-formula}, we discuss some lemmas concerning the local cohomology modules $\h_{\mathcal{M}}^i(G(I))$ of $G(I)$ with support in
$\mathcal{M}$ for an ideal $I$.
These results mainly develop the setting for the proofs of our main results in
subsequent sections. The content of this section is very basic but we include them for clarity.

In Section \ref{bounding-the-hilbert-coefficients}, we prove that $e_d(Q)\leq 0$ for a parameter ideal $Q$
if $[\h_{\mathcal{M}}^i(G(Q))]_n=0$ for all $n\leq -1$ and $0\leq i\leq d-3$ (Theorem \ref{theorem-main-(-1)}). Note that
this condition holds if $\depth G(Q)\geq d-2.$ As a consequence, we prove the following theorem.
\begin{theorem}
   Let $(R,\m)$ be a Noetherian local ring of dimension $d\geq 3$ and $\depth R\geq d-1$. Let $Q$ be a parameter ideal. Then $e_3(Q)\leq 0.$
\end{theorem}

For the higher coefficients $e_i(Q)$, we are able to prove the following result.
\begin{theorem}
  Let $(R,\m)$ be a Noetherian local ring of dimension $d\geq 2$ and $\depth R\geq d-1$. Let $Q$ be a parameter ideal of $R$.
  \begin{enumerate}
   \item Suppose $\depth(G(Q))\geq d-2$. Then $ e_i(Q)\leq 0$ for $2\leq i\leq d$.
   \item Suppose $\depth(G(Q))\geq d-1$. Then $ e_i(Q)\geq -\lambda_R(\h_{\m}^{d-1}(R))$ for $2\leq i\leq d$.
  \end{enumerate}
\end{theorem}

In Section \ref{vanishing-of-coefficients}, we discuss the vanishing of the last coefficient $e_d(Q)$. In \cite[Theorem 3.2]{goto-ozeki}, Goto and Ozeki found a necessary and sufficient condition
for the vanishing of $e_2(Q)$ in rings of dimension two and depth at least one. We generalize their result in two directions. In Theorem \ref{theorem-vanishing-of-ed}, we find a necessary condition for the
equality $e_d(Q)=0$ provided $\depth R\geq d-1$. Let $Q=(x_1,\ldots,x_d)$ with
$x_1^*,\ldots,x_{d-1}^*$ a superficial sequence. Suppose $\depth G(Q)\geq d-2$ and $\h_{\mathcal{M}}^{d-2}(G(Q))$ is finitely graded, then we prove that
$e_d(Q)=0$ implies $x_1^l,\ldots,x_{d-1}^{l},x_d^{(d-1)l}$ is a $d$-sequence for all integers $l\geq 1$. Further in Theorem \ref{theorem-e2}, we extend \cite[Theorem 3.2]{goto-ozeki}
to dimension $d\geq 2$ by providing similar equivalent conditions for the vanishing of $e_2(Q)$. Our result on $e_2(Q)$ also includes \cite[Theorem 3.2]{Lori}.

We refer to \cite{bruns-herzog} for undefined terms.
\section{Preliminary results}\label{section-formula}
% % Note that the assumption $b_i(G(Q))\geq 0$ for $0\leq i\leq d-2$ is weaker than the assumption $\depth G(Q)\geq d-1$.
In this section, we discuss few lemmas which are the key steps for the results of subsequent sections.
% % \sal{add one line about the results of this section}
% % The main motivation for framing these lemmas is to be able to avoid the higher local cohomology modules wherever we can.
For this section, let $I\subseteq \m$ be an arbitrary ideal of $R$.  We have
% % We frequently use the fact that for all $i\in\ints$,
$$H_{\mathcal{M}}^i(G(I))=H_{\mathcal{R}_{+}}^i(G(I))=H_{G_{+}}^i(G(I)).$$
It is well known that there exists $m_I\in\ints$  such that \begin{equation}\label{eqn-6-09}
                  [\h_{\mathcal{M}}^i(G(I))]_{n+1}=0 \text{ for all } n\geq m_I \text{ and for all } i\in\ints.
                 \end{equation}
For an element $0\neq x\in R$,
let $x^*$ denote the  initial form of $x$ in $G(I)$, i.e.
the image of $x$ in $G(I)_i$  where $i$ is the unique integer such that $x\in I^i\setminus I^{i+1}$.

% % for some integer $m_I.$
\begin{lemma}\label{lemma-3}
Let $R$ be a Noetherian local ring and $I$ an ideal of $R$.
Let $x_1,\ldots,x_r\in I\setminus I^2$ such that $x_1^*,\ldots,x_r^*$ is a
regular sequence in $G(I)$. Then, for $1\leq j\leq r$,
\begin{equation*}
 [\h_{\mathcal{M}}^j(G(I))]_{m_I}\cong  [\h_{\mathcal{M}}^{j-1}(G(IR_1))]_{{m_I}+1}\cong\ldots\cong [\h_{\mathcal{M}}^0(G(IR_j))]_{m_I+j}
\end{equation*}
% where $I_j=I/(x_1,\ldots,x_j)\subseteq R/(x_1,\ldots,x_j)$. Moreover, for $n\geq m+j$ and $i\in\ints$,
where $R_j=R/(x_1,\ldots,x_j)$. Moreover, for all $n\geq m_I+j$ and for all $i\in\ints$,
\begin{equation*}
[\h_{\mathcal{M}}^i(G(IR_j))]_{n+1}=0.
\end{equation*}
\end{lemma}
\begin{proof}
 We apply induction on $r$.  Consider the following exact sequence and the induced
 long exact sequence of local cohomology modules.
 \begin{align}
  0&\to G(I)(-1)\xrightarrow{x_1^*} G(I)\to G(IR_1)\to 0\nonumber\\
  \ldots\to \h_{\mathcal{M}}^i(G(I))&\to \h_{\mathcal{M}}^i(G(IR_1))\to  \h_{\mathcal{M}}^{i+1}(G(I))(-1)\to \h_{\mathcal{M}}^{i+1}(G(I))
  \to\ldots\label{long-exact-seq-1}
  \end{align}

 Since $[\h_{\mathcal{M}}^i(G(I))]_{n+1}=0$ for all $n\geq m_I$ and $i\in\ints$, we get from \eqref{long-exact-seq-1} that  for all $i\in\ints$,
 \begin{align}
[\h_{\mathcal{M}}^i(G(I))]_{m_I}\cong [\h_{\mathcal{M}}^{i-1}(G(IR_1))]_{m_I+1}\text{ and }\label{induction-eqn-1}\\
   [\h_{\mathcal{M}}^i(G(IR_1))]_{n+1}=0 \text{ for } n\geq m_I + 1.\nonumber%\label{induction-eqn-2}
 \end{align}
 By putting $i=1$ in \eqref{induction-eqn-1}, we get the result for $r=1$.

  Suppose $r>1$ and the assertion holds for $r-1$.
 Since $x_1^*,\ldots,x_r^*$ is a regular sequence in $G(I)$, we have $G(IR_1)\cong G(I)/x_1^*G(I)$ and $x_2^*,\ldots,x_r^*$ is
a regular sequence in  $G(IR_1)$. By induction hypothesis, for $1\leq k\leq r-1$,
 \begin{gather}
[\h_{\mathcal{M}}^k(G(IR_1))]_{m_I+1}\cong  [\h_{\mathcal{M}}^{k-1}(G(IR_2))]_{m_I+2}\cong \ldots\cong [\h_{\mathcal{M}}^0(G(IR_{k+1}))]_{m_I+1+k}\text{ and }\label{induction-eqn-3}\\
\text{[}\h_{\mathcal{M}}^i(G(IR_{k+1}))]_{n+1}=0 \text{ for all } n\geq m_I+1+k \text{ and for all } i\in\ints. \label{induction-eqn-4}
 \end{gather}
Now  let $1\leq j\leq r$.
% % If $j=1$, then the results follows from Lemma \ref{lemma-2}.
We may assume that $j>1$. Then  \eqref{induction-eqn-1}, \eqref{induction-eqn-3}
and \eqref{induction-eqn-4} with $k=j-1$ give
\begin{gather*}
  [\h_{\mathcal{M}}^j(G(I))]_{m_I}\cong  [\h_{\mathcal{M}}^{j-1}(G(IR_1))]_{m_I+1}\cong\ldots\cong [\h_{\mathcal{M}}^0(G(IR_j))]_{m_I+j}\text{ and }\\
 \text{[}\h_{\mathcal{M}}^i(G(IR_j))]_{n+1}=0  \text{ for all } n\geq m_I+j \text{ and for all } i\in\ints.
  \end{gather*}
\end{proof}
The next lemma relates the local cohomology of Rees algebra and the associated graded ring.
\begin{lemma}\label{lemma-1}
% Suppose there exists an integer $m\geq 1$ such that
% % Let $[\h_{\mathcal{R}_{+}}^i(G(I))]_n=0$ for all $n\geq m$ and $i\in\ints$. Then
Let $R$ be a Noetherian local ring and $I$ an ideal of $R$. Then
$$[\h_{\mathcal{R}_{+}}^i(\mathcal{R})]_n\cong [\h_{\mathcal{R}_{+}}^i(G(I))]_n$$ for all   $n>\max\{ m_I-1,-1\}$ and for all $i\in\ints$.
\end{lemma}
\begin{proof}
% This proof is copied from \cite[Lemma 3.1]{goto-ozeki}.
 Consider the following exact sequences with the canonical maps
 $$0\to \mathcal{R}_{+}\to \mathcal{R}\to R\to 0\text{ and } 0\to \mathcal{R}_{+}(1)\to \mathcal{R}\to G(I) \to 0$$
 and apply the functor $H^i_{\mathcal{R}_{+}}(*)$ to get
 \begin{gather}
  \hspace{-3cm}\ldots\to H^{i-1}_{\mathcal{R}_{+}}(R)\to H^i_{\mathcal{R}_{+}}(\mathcal{R}_{+})\to H^i_{\mathcal{R}_{+}}(\mathcal{R})\to H^{i}_{\mathcal{R}_{+}}(R)\to \ldots\label{exact-seq1}\\
 \to H^{i-1}_{\mathcal{R}_{+}}(G(I))\to H^{i}_{\mathcal{R}_{+}}(\mathcal{R}_{+})(1)\to H^{i}_{\mathcal{R}_{+}}(\mathcal{R})\to
  H^{i}_{\mathcal{R}_{+}}(G(I))\to H^{i+1}_{\mathcal{R}_{+}}(\mathcal{R}_{+})(1)\to\label{exact-seq2}
 \end{gather}
Since $[H^{i}_{\mathcal{R}_{+}}(G(I))]_n=0$ for all $i\in\ints$ and for all $n>m_I$, we get from exact sequence \eqref{exact-seq2} that
\begin{equation}
[H^{i}_{\mathcal{R}_{+}}(\mathcal{R}_{+})]_{n+1}\cong[H^{i}_{\mathcal{R}_{+}}(\mathcal{R})]_n\label{eqn-chap-last-e1}
\end{equation}
for all $n> m_I$ and for all $i\in\ints$. Further by exact sequence \eqref{exact-seq1}, we have
\begin{equation}[H^{i}_{\mathcal{R}_{+}}(\mathcal{R}_{+})]_{n+1}\cong [H^{i}_{\mathcal{R}_{+}}(\mathcal{R})]_{n+1}\label{eqn-chap-last-e1-17/10-1}\end{equation}
for all  $n\geq 0$ and $i\in\ints$. By \eqref{eqn-chap-last-e1} and \eqref{eqn-chap-last-e1-17/10-1}, 
$[H^{i}_{\mathcal{R}_{+}}(\mathcal{R})]_n\cong [H^{i}_{\mathcal{R}_{+}}(\mathcal{R})]_{n+1}$ for all 
 $n>\max\{m_I,0\}$.
Since $[H^{i}_{\mathcal{R}_{+}}(\mathcal{R})]_n=0$ for all $n\gg 0$,
$[H^{i}_{\mathcal{R}_{+}}(\mathcal{R})]_n=0$ for all  $n>\max\{m_I,0\}$. Therefore 
by \eqref{eqn-chap-last-e1-17/10-1}, $[H^{i}_{\mathcal{R}_{+}}(\mathcal{R}_{+})(1)]_{n}=0$
for all   $n> \max\{m_I-1,-1\}$ and $i\in\ints$. Now by exact sequence \eqref{exact-seq2}, we get that
$$[\h_{\mathcal{R}_{+}}^i(\mathcal{R})]_n\cong [\h_{\mathcal{R}_{+}}^i(G(I))]_n$$
for all  $n> \max \{m_I-1,-1\}$ and for all $i\in\ints$.
\end{proof}
We set
\begin{align*}
 a_i\left(G(I)\right)&:= \sup\{n\in\ints: [\h_{\mathcal{R_+}}^i(G(I))]_n\neq0\}  \mbox{ and }\\
 b_i\left(G(I)\right)&:= \inf\{n\in\ints: [\h_{\mathcal{R_+}}^i(G(I))]_n\neq0\}.
\end{align*}
By convention, if $\h_{\mathcal{R_+}}^i(G(I))=0$ then we set $a_i(G(I))=-\infty$ and $b_i(G(I))=\infty$. Note that 
$b_0(G(I))\geq 0$.
The next lemma plays crucial role in most of our proofs.
 We show that given a parameter ideal $Q$ with $b_i(G(Q))\geq 0$ for $0\leq i\leq d-3,$
 $G(Q^l)$ has high depth for all $l\gg 0$.

 Recall that a {\it reduction} of an ideal $I$ is an ideal $J\subseteq I$ such that $I^{n+1}=JI^n$ for some $n\geq 0.$ A {\it minimal reduction} of $I$ is a reduction of $I$ which is minimal with respect to inclusion.
For a minimal reduction $J$ of $I$, {\it reduction number} of $I$ with respect to $J$, denoted by $r_J(I)$, is the least non-negative integer $n$ such that $I^{n+1}=JI^n$.
For an ideal $I$, let $\mu(I)$ denote the minimal number of generators of $I$.
\begin{lemma}\label{lemma-4}
  Let $R$ be a Noetherian local ring of dimension $d\geq 2$ and
$\depth R\geq d-1$. Let $Q$ be parameter ideal such that $b_i(G(Q))\geq 0$ for $0\leq i\leq d-3$. Then the following assertions hold.
 \begin{enumerate}[label=(\alph*)]
 \item \label{lemma-4-item-0} For all $l\gg 0$,
 \begin{align}\label{31-aug-6-e1}
  a_i(G(Q^l)))\leq \begin{cases}
      0 &\text{ for all } i\in\ints ,\\
      -1 & \text{ for } i=d, d-2
 \end{cases}
\end{align}
 Furthermore, $\h_{\mathcal{M}}^i(G(Q^l))=0$  for  $0\leq i\leq d-3.$ In particular, $\depth G(Q^l)\geq d-2$ for $l\gg 0.$

\item \label{lemma-4-item-3} Suppose $\h_{\mathcal{M}}^{d-2}(G(Q))$ is finitely graded and  $[\h_{\mathcal{M}}^{d-1}(G(Q^{l}))]_0=0$ whenever $l$ is sufficiently large. Then $\depth G(Q^{l_0})\geq d-1$ for any sufficiently large integer $l_{0}$.
\end{enumerate}
\end{lemma}
\begin{proof}
We may assume that the residue field $R/\m$ is infinite.
 Let $Q=(x_1,\ldots,x_d)$ such that $x_1^*,\ldots,x_d^*$ is a superficial sequence. For $l\geq 1$, we put $I=Q^l$.
  For a rational number $a$, let $\lfloor a \rfloor=\max\{n\in\ints: n\leq a\}$.
Given that  $b_i(G(Q))\geq 0$ for $0\leq i\leq d-3$. Thus $b_i(G(I))\geq \lfloor b_i(G(Q))/l \rfloor \geq 0$ for $0\leq i\leq d-3$ by
\cite[Lemma 2.4]{Hoa}. Choose  $l>\max\{|a_i(G(Q))|:a_i(G(Q))\neq -\infty\}$
% % % $l>\max\{\lfloor a_i(G(Q))\rfloor:a_i(G(Q))\neq -\infty\}$ 
and $y_i=x_i^l$ for $1\leq i\leq d$. Then $y_1^*,\ldots,y_d^*$ is a superficial sequence with respect to
$I$ and 
% % % % %   This is trivial
\begin{equation}\label{eqn-reduction-no-of-I}
 I^{d}=(y_1,\ldots,y_d)I^{d-1}.
\end{equation}
 To see the above equality, note that $I^d=Q^{ld}$ is generated by monomials in $x_1,\ldots,x_d$ of degree $ld$. 
Let $m=x_1^{t_1}\ldots x_d^{t_d}$ with $t_1+\ldots+t_d=ld$ be a generator of $I^d$. Then at least one $t_i\geq l$, so 
$m=x_i^{l}.m'$ where $m'=x_1^{t_1}\ldots x_i^{t_i-l}\ldots x_d^{t_d}\in Q^{ld-l}$. Hence 
$m\in (x_1^l,\ldots,x_i^l,\ldots,x_d^l)Q^{l(d-1)}=(y_1,\ldots,y_d)I^{d-1}$ which gives 
$I^d \subseteq(y_1,\ldots,y_d)I^{d-1}$.
 Since $J=(y_1,\ldots,y_d)$ is a reduction of $I$, $\mu(J)= d$. Hence $J$ is a minimal reduction of $I$ with $r_J(Q^l)\leq d-1$. So by
 \cite[Proposition 3.2]{trung-reduction}, $a_d(G(I))\leq r_J(I)-d.$
% % Theorem \ref{hoa-theorem-2.1}, $a_d(G(Q))<0$.

$\ref{lemma-4-item-0}$ It follows that $a_d(G(I))<0$ and  $a_i(G(I))\leq \lfloor a_i(G(Q))/l\rfloor \leq 0$ for $i\leq d-1$ by choice of $l$ and \cite[Lemma 2.4]{Hoa}.
In other words,
 \begin{align}\label{equation:bi(G(I))}
  [\h_{\mathcal{M}}^i(G(I))]_n=0\begin{cases}
                                  \text{ for all } n\geq 1 \text{~~and } i\in \ints,\\
                                  \text{ for all } n\neq 0 \text{~~and } 0\leq i\leq d-3,\\
                                  \text{ for all } n\geq 0 \text{~~and } i=d.
                              \end{cases}
 \end{align}
%  \eqref{lemma-4-item-1}
 \textbf{Claim:} $[\h_{\mathcal{M}}^{i}(G(I))]_0=0 \text{ for } 0\leq i\leq d-2$.

%  \begin{proof}[Proof of Claim]
 {\it {Proof of Claim.}}
   We apply induction on $i$. For $i=0,$
% %  We apply induction on $d.$ For $d=2,$
 $[\h_{\mathcal{M}}^0(G(I))]_0\cong[\h_{\mathcal{R}_+}^0(R(I))]_0=0$ by Lemma \ref{lemma-1} as $\depth R\geq 1$.
% %  In view of \eqref{equation:bi(G(I))}, it is now enough to
% % Let $d\geq 3$ and 
   Let $[\h_{\mathcal{M}}^{i}(G(I))]_0=0$ for $0\leq i\leq s$ for some $s\leq d-3$. Using \eqref{equation:bi(G(I))},
we get that $\h_{\mathcal{M}}^{i}(G(I))=0$ for $0\leq i\leq s$. So $\depth (G(I))\geq s+1$ and $y_1^*,\ldots,y_{s+1}^*$ is a regular sequence in $G(I)$. By
Lemma \ref{lemma-3},
\begin{align}
[\h_{\mathcal{M}}^{s+1}(G(I))]_0 &\cong [\h_{\mathcal{M}}^0(G(IR_{s+1}))]_{s+1} \text{ and }\label{eqn-chap-last-e2}\\
{[}\h_{\mathcal{M}}^i(G(IR_{s+1}))]_{n} &= 0 \text{ for all } n\geq s+2\text{ and }i\in\ints\nonumber
\end{align}
where $R_{s+1}=R/(y_1,\ldots,y_{s+1})$.
Thus using Lemma \ref{lemma-1} with $m_{IR_{s+1}}=s+1$, we get $[\h_{\mathcal{M}}^0(G(IR_{s+1}))]_{s+1}\cong [\h_{\mathcal{R}_+}^0(R(IR_{s+1}))]_{s+1}=0$
since $\depth R_{s+1}\geq 1$.
Thus $[\h_{\mathcal{M}}^{s+1}(G(I))]_0=0$ by \eqref{eqn-chap-last-e2}. This completes the proof of the claim. $\square$

By above claim, $a_{d-2}(G(Q^l))\leq -1$ and $\h_{\mathcal{M}}^i(G(I))=0$  for  $0\leq i\leq d-3$ which implies that $\depth G(Q^l)\geq d-2$ for all $l\gg 0$.

%%%%%%%%%%%%%%%%%%%%%%%%%%%%%%%%%%%%%%%%%%%%%%%%%%%%%%%%%%%%%%%%%%% 26/10/16
\ref{lemma-4-item-3}   For $d=2$ case, $a_{d-2}(G(Q^l))\leq -1$ implies $[\h_{\mathcal{M}}^{0}(G(I))]_n=0$ for 
all $n\geq 0$.
Thus $\h_{\mathcal{M}}^{0}(G(I))=0$ and $\depth G(I)\geq 1.$ Now let $d\geq 3.$
We can choose $l_{0}> \max \{|b_{d-2}(G(Q))|, |a_i(G(Q))|:a_i(G(Q))\neq -\infty\}$ such that
$[\h_{\mathcal{M}}^{d-1}(G(Q^{l_0}))]_0=0$. By part $(a)$,  $\h_{\mathcal{M}}^{i}(G(Q^{l_0}))=0$ for $0\leq i \leq d-3$. If 
we show that $\h_{\mathcal{M}}^{d-2}(G(Q^{l_0}))=0$, it will follow that $\depth G(Q^{l_0})\geq d-1.$ Let $I_{0}=Q^{l_0}$.
%%%%%%%%%%%%%%%%%%%%%%%%%%%%%%%%%%%%%%%%%%%%%%%%%%%%%%%%%%%%%%%%%%% 26/10/16
% By part \eqref{lemma-4-item-0}, $\depth G(Q^l)\geq d-2.$
% % which implies $\depth G(Q^{l_0})\geq d-1.$
  Since $\h_{\mathcal{M}}^{d-2}(G(Q))$ is finitely graded, 
 $b_{d-2}(G(I_0))\geq \lfloor b_{d-2}(G(Q))/l_0 \rfloor \geq -1$ by \cite[Lemma 2.4]{Hoa}. Thus, by part \ref{lemma-4-item-0},
% % and \eqref{equation:bi(G(I))} \sal{better},
\begin{equation*}
[\h_{\mathcal{M}}^{d-2}(G(I_0))]_n=0 \mbox{ for } n\neq -1.
\end{equation*}
Now it is enough to show that $[\h_{\mathcal{M}}^{d-2}(G(I_0))]_{-1}=0$.
Given that $[\h_{\mathcal{M}}^{d-1}(G(I_0))]_0=0$. By part \ref{lemma-4-item-0} and \eqref{equation:bi(G(I))}, we get that $[\h_{\mathcal{M}}^{i}(G(I_0))]_n=0$ for all $n\geq 0$
and $i\in\ints$. Also $\depth G(I_0)\geq d-2$ implies that $y_1^*,\ldots,y_{d-2}^*$ is a regular sequence in $G(I_0).$ Therefore by Lemma \ref{lemma-3},
\begin{align*}
[\h_{\mathcal{M}}^{d-2}(G(I_0))]_{-1}&\cong[\h_{\mathcal{M}}^{0}(G(I_0R_{d-2}))]_{d-3} \mbox{ and}\\
{[}\h_{\mathcal{M}}^{i}(G(I_0R_{d-2}))]_{n+1}&=0 \mbox{ for } n\geq d-3 \mbox{ and } i\in\ints.
\end{align*}
 So by Lemma \ref{lemma-1}, $[\h_{\mathcal{M}}^{0}(G(I_0R_{d-2}))]_{d-3}\cong [\h_{\mathcal{M}}^{0}(R(I_0R_{d-2}))]_{d-3}=0$
as $\depth R_{d-2}\geq 1$.  Hence $[\h_{\mathcal{M}}^{d-2}(G(I_0))]_{-1}=0$.
% % \end{enumerate}
 \end{proof}

%%%%%%%%%%%%%%%%%%%%%%%%%%%%%%%%%%%%%%%%%%%%%%%%%%%%%%%%%%%%%%%%%%%%%%%%%%%%%%%
\section{Bounding the Hilbert coefficients}\label{bounding-the-hilbert-coefficients}
In this section, we obtain bounds on the coefficients $e_i(Q)$ for a parameter ideal $Q$ in a ring of depth at least $d-1$ with certain conditions on the local
cohomology modules of $G(Q)$. We show that the last coefficient $e_d(Q)\leq 0$ if $b_i(G(Q))\geq 0$ for $0\leq i\leq d-3$ (Theorem \ref{theorem-main-(-1)}) and that
$e_d(Q)\geq-\lambda_R(\h_{\m}^{d-1}(R))$ if $b_i(G(Q))\geq 0$ for $0\leq i\leq d-2$ (Theorem \ref{theorem-main-2}). Note that the above conditions on $b_i(G(Q))$
holds if $\depth G(Q)$ is at least $d-2$ and $d-1$ respectively. Consequently for $2\leq i\leq d $, we obtain that (1) $e_i(Q)\leq 0$ provided $\depth G(Q)\geq d-2$
(Corollary \ref{corr-ei-lesseq-0}) and (2) $e_i(Q)\geq -\lambda_R(\h_{\m}^{d-1}(R))$ provided $\depth G(Q)\geq d-1$ (Corollary \ref{corr-ei-grteq}).

The most interesting result of this section is Theorem \ref{theorem-e3} which states that $e_3(Q)\leq 0$ for a parameter ideal $Q$ with $\depth R\geq d-1$.
In order to prove this, we first need the  non-positivity of $e_d(Q)$.
 \begin{theorem}\label{theorem-main-(-1)}
 Let $(R,\m)$ be a Noetherian local ring of dimension $d\geq 2$ and $\depth R\geq d-1$. Let $Q$ be a parameter ideal such that
 $b_i(G(Q))\geq 0$ for $0\leq i\leq d-3.$ Then
 $$e_d(Q)\leq 0.$$
\end{theorem}
\begin{proof}
We may assume that the residue field is infinite. Let $Q=(x_1,\ldots,x_d)$ such that $x_1^*,\ldots,x_d^*$ is a superficial sequence.
For an integer $l\gg 0$, we put $I=Q^l$. By Lemma \ref{lemma-4}$\ref{lemma-4-item-0}$,  \begin{equation}\label{equation:vanishing-of-n-pieces}
                                                                                        [\h_{\mathcal{M}}^{i}(G(I))]_0=0 \text{ for } 0\leq i\leq d-2 \mbox{ and } i=d.
                                                                                       \end{equation}
and by \cite[Theorem 3.8]{blancafort}, $[\h_{\mathcal{R}_+}^i(R(I)^*)]_0\cong [\h_{\mathcal{R}_+}^i(R(I))]_0$ for all $i\geq 0$. Thus
% % \cite[Theorem 3.8]{blancafort}.
\cite[Theorem 4.1]{blancafort} yields that
 \begin{align}
(-1)^d e_d(I)&=P(I,0)- H(I,0)\nonumber\\
&=\sum_{i=0}^d(-1)^i\lambda_R([\h_{\mathcal{R}_+}^i(R(I)^*)]_0)\nonumber\\
&=\sum_{i=0}^d(-1)^i\lambda_R([\h_{\mathcal{R}_+}^i(R(I))]_0).\label{formula-for-ed}
\end{align}
% % We may choose $l>max\{|a_i(G(Q))|:a_i(G(Q))\neq -\infty\}$.
By Lemma \ref{lemma-4}$\ref{lemma-4-item-0}$, $[\h_{\mathcal{M}}^i(G(I))]_n=0$ for all $n\geq 1$ and $i\in \ints$.
 Therefore by Lemma \ref{lemma-1} and \eqref{formula-for-ed}, we get
 \begin{align*}
 (-1)^de_d(I)=\sum_{i=0}^d(-1)^i\lambda_R([\h_{\mathcal{M}}^i(G(I))]_0)=(-1)^{d-1}\lambda_R([\h_{\mathcal{M}}^{d-1}(G(I))]_0)
\end{align*}
where the last equality holds due to \eqref{equation:vanishing-of-n-pieces}.
% % and the fact that $[\h_{\mathcal{M}}^{d}(G(I))]_0=0$ from \eqref{31-aug-6-e1}.
This implies
\begin{equation}\label{equation:ed-negative}
 e_d(Q)=e_d(I)=-\lambda_R([\h_{\mathcal{M}}^{d-1}(G(I))]_0)\leq 0.
\end{equation}
\end{proof}
In \cite[Corollary 4.5]{Lori}, Mccune proved that the coefficients
$e_i(Q)$, for $2\leq i\leq d$, are all non-positive when $\depth G(Q)\geq d-1$. In the next corollary, we improve upon her result by weakening the hypothesis
to $\depth G(Q)\geq d-2$.
\begin{corollary}\label{corr-ei-lesseq-0}
  Let $(R,\m)$ be a Noetherian local ring of dimension $d\geq 2$ and $\depth R\geq d-1$. Let $Q$ be a parameter ideal of $R$
  such that $\depth(G(Q))\geq d-2$. Then, for $2\leq i\leq d$,
%  $-\lambda(\h_{\m}^{d-1}(R))\leq e_i(Q)\leq 0$ for $1\leq i\leq d$.
 \begin{equation*}
 e_i(Q)\leq 0.
  \end{equation*}
\end{corollary}
\begin{proof}
We may assume that the residue field $R/\m$ is infinite. Let $Q=(x_1,\ldots,x_d)$ such that  $x_1^*,\ldots,x_d^*$ is a superficial sequence.
Set $R_0=R$ and $R_i=R/(x_1,\ldots,x_i)$ for $1\leq i\leq d-2$.
Since $x_1^*,\ldots,x_{d-2}^*$ is a regular sequence in $G(Q)$, $G(QR_i)\cong G(Q)/(x_1^*,\ldots,x_i^*)G(Q)$ and $\depth G(QR_i)\geq d-i-2$ for $i\leq d-2$.
% Now reducing by superficial elements $x_1,\ldots,x_i$ and
Hence $e_i(Q)=e_i(QR_{d-i})\leq 0$ for $2\leq i\leq d$ by Theorem \ref{theorem-main-(-1)}.
\end{proof}
 The following example emphasizes that the depth condition on the ring is necessary in the above corollary. This example is motivated by the idea presented in \cite{goto-ozeki}.
  \begin{example}\label{example-gcm-22-1}
   Let $(R,\mathfrak n)$ be a regular local ring of dimension $d\geq 2$ and $X_1,\ldots,X_d$ a regular system of parameters of $R$.
   We put $\mathfrak p_t=(X_1,\ldots,X_{d-t})$ for some $1\leq  t\leq d-1$ and $D=R/\mathfrak{p_t}.$ Let $A=R\ltimes D$ be the idealization
   of $D$ over $R.$ Then $A$ is a Noetherian local ring dimension $d$  with the maximal ideal $\m=\mathfrak{n}\times D$ and $\depth A=t$.
   Consider the exact sequence of $A$-modules
   \begin{equation}\label{exact-seq-22-1}
    0\to D\xrightarrow{j} A \xrightarrow{p} R\to 0
   \end{equation}
where $j(x)=(0,x)$ and $p(a,x)=a$. Note that $D$ is an $A$-module via $p$.
Let $Q$ be a parameter ideal in $A$ and $q=p(Q)\subseteq R.$ Then
we have,
\begin{align}
 \lambda_A(A/Q^{n+1})&=\lambda_R(R/q^{n+1})+\lambda_R(D/q^{n+1}D)\nonumber\\
 &=e_0(q,R)\binom{n+d}{d}+\sum_{i=0}^t(-1)^ie_i(q,D)\binom{n+t-i}{t-i}\nonumber
\end{align}
for all $n\gg 0.$ This implies
 \begin{align}\label{eqn-22-1}
e_i(Q,A)=\begin{cases}
        e_0(q,R) & \text{ if }\ i=0\\
          0 & \text{ if }\ 1\leq i\leq d-t-1\\
          (-1)^{d-t}e_{i-d+t}(q,D) & \text{ if}\ d-t\leq i\leq d.
         \end{cases}
 \end{align}
In particular, let $d=4$ and $t=2$ so that
$D=R/(X_1,X_2)$. Let $q=(X_1,\ldots,X_d)$ and $Q=qA$. Then $\dim A=4$, $\depth A=2.$
Since $G(Q)=G(q)\ltimes G(qD)(-1)$, see \cite[Remark 2]{tony}, we have $\depth G(Q)=2$ but $e_2(Q,A)>0$ by \eqref{eqn-22-1}.
\end{example}
A noteworthy consequence of Theorem \ref{theorem-main-(-1)} is the following result.
\begin{theorem}\label{theorem-e3}
   Let $(R,\m)$ be a Noetherian local ring of dimension $d\geq 3$ and $\depth R\geq d-1$. Let $Q$ be a parameter ideal. Then $e_3(Q)\leq 0.$
\end{theorem}
\begin{proof}
 We may assume that $R/\m$ is infinite. Then using reduction modulo superficial elements, it is enough to assume that $d=3$.
 The result now follows from Theorem \ref{theorem-main-(-1)}.
\end{proof}
The following example shows that the assumption on the depth of the ring can not be relaxed from Theorem \ref{theorem-e3}.
\begin{example}\label{example-22-1-chap-6}{\rm\cite[Example 4.7]{goto-ozeki-2}}
 Let $(S,\mathfrak{n})$ be a regular local ring of dimension $d=4$ with infinite residue field $S/\mathfrak{n}$. Let $X,Y,Z,W$ be a regular system
 of parameters of $S$ and $R=S/((X)\cap (Y^3,Z,W))$. Let $x,y,z,w$ be the images of $X,Y,Z,W$ in $R$ respectively and
 $\m=(x,y,z,w)R$ be the maximal ideal of $R$. Then $\dim R=3,$ $\depth R=1$. Let $U=(x)$, $Q=(x-y,x-z,x-w)R$ and $T=R/(x)$.
 Then $T$ is a regular local ring with $\dim T=3$ and $QT=\m T$. The following exact sequence
 $$0\to (x)\to R\to R/(x)\to 0$$ gives that
 \begin{align}
  \lambda_R(R/Q^{n+1}R)&=\lambda_R(T/\m^{n+1}T)+\lambda_R(U/Q^{n+1}U)\nonumber\\
  &=\binom{n+3}{3}+e_0(Q,U)\binom{n+1}{1}-e_1(Q,U)\label{eqn-23-1-example}
 \end{align}
for all $n\gg 0.$ We have
 $(x)\cong R/I$ where $I=(y^3,z,w)R$ and $Q(R/I)=\m (R/I)$, so $e_0(Q,U)=e_0(\m,R/I)$ and $e_1(Q,U)=e_1(\m,R/I)$. The
 Hilbert series of the associated graded ring $G(\m (R/I))$ is
$$\frac{1+t+t^2}{1-t}$$
Hence  $e_0(Q,U)=e_0(\m,R/I)=3$ and $e_1(Q,U)=e_1(\m,R/I)=3.$ By \eqref{eqn-23-1-example}, we get
 $e_3(Q,R)=3>0.$
\end{example}
The following lemma is crucial for obtaining lower bound on $e_i(Q)$. We also obtain a necessary condition for the vanishing of $e_d(Q)$ in
Theorem \ref{theorem-vanishing-of-ed} as an application of this lemma.
% % We treat the associated graded ring $G(I)$ as the quotient of the Rees algebra $R(I)/IR(I).$
\begin{lemma}\label{lemma-main-1} Let $(R,\m)$ be a Noetherian local ring of dimension $d\geq 2$ and $\depth R\geq d-1$. Let $I$ be an $\m$-primary ideal such
that
$\depth G(I)\geq d-1$ and
 \begin{align}\label{eqn-lemma-main-1}
  a_i(G(I))\leq \begin{cases}
                 0 & \text{ for } i\in\ints, \\
		  -1 & \text{ for } i=d.
                \end{cases}
 \end{align}
Let $J=(y_1,\ldots,y_d)$ be a reduction of $I$ with $I^{d}=JI^{d-1}$ and $y_1^*,\ldots,y_{d-1}^*$
 is a superficial sequence. Then
 \begin{equation}\label{eqn-main-formula-for-ed}
  e_d(I)=-\lambda_R\left(\frac{((y_1,\ldots,y_{d-1}):y_d)\cap (I^{d-1}+(y_1,\ldots,y_{d-1}))}{(y_1,\ldots,y_{d-1})}\right).
 \end{equation}
\end{lemma}
\begin{proof}
Since $\depth G(I)\geq d-1$, $y_1^*,\ldots,y_{d-1}^*$ is a regular sequence in $G(I)$ and $\h_{\mathcal{M}}^i(G(I))=0$ for $0\leq i\leq d-2$. By Lemma \ref{lemma-1},
$[\h_{\mathcal{M}}^i(G(I))]_0\cong [\h_{\mathcal{R}_+}^i(R(I))]_0$ for all $i$. Therefore using \eqref{formula-for-ed},
\begin{align*}
 (-1)^de_d(I)&=\sum_{i=0}^d(-1)^i\lambda_R([\h_{\mathcal{R}_+}^i(R(I))]_0)\\
 &=\sum_{i=0}^d(-1)^i\lambda_R([\h_{\mathcal{M}}^i(G(I))]_0)
% %  &=&-\lambda_R([\h_{\mathcal{M}}^{d-1}(G(I))]_0)
\end{align*}

Using \eqref{eqn-lemma-main-1} and Lemma \ref{lemma-3} respectively, we get that
\begin{equation}
e_d(I)=-\lambda_R([\h_{\mathcal{M}}^{d-1}(G(I))]_0)=-\lambda_R([\h_{\mathcal{M}}^0(G(IR_{d-1}))]_{d-1}) \label{eqn-18/10}
\end{equation}
where $R_{d-1}=R/(y_1,\ldots,y_{d-1})$.
Now consider the map $$\rho : \frac{((y_1,\ldots,y_{d-1}):y_d)\cap (I^{d-1}+(y_1,\ldots,y_{d-1}))}{(y_1,\ldots,y_{d-1})} \to [\h_{\mathcal{M}}^0(G(IR_{d-1}))]_{d-1}$$
defined as $\rho(\bar{x})=\overline{\bar{x}t^{d-1}}$ where $\bar{x}$ and $\overline{\bar{x}t^{d-1}}$ are the images of $x\in R$ in $R/(y_1,\ldots,y_{d-1})$ and
$\bar{x}t^{d-1}\in R(IR_{d-1})$ in $G(IR_{d-1})$ respectively. It is now enough to show that $\rho$ is an isomorphism. To show surjectivity, let $\alpha=\overline{\bar{x}t^{d-1}}\in [\h_{\mathcal{M}}^0(G(IR_{d-1}))]_{d-1}$  with
$x\in I^{d-1}$. Then
\begin{equation}\label{eqn-pho-surjective-1}
 \bar{y}_dt~\cdot~\overline{\bar{x}t^{d-1}}=\overline{\overline{y_dx}t^d}\in [\h_{\mathcal{M}}^0(G(IR_{d-1}))]_{d}.
\end{equation}
Since $[\h_{\mathcal{M}}^0(G(IR_{d-1}))]_{d}=0$  by Lemma \ref{lemma-3} and
$I^d=JI^{d-1}$,
\eqref{eqn-pho-surjective-1} yields that
\begin{equation*}
y_dx\in(I^{d+1}+(y_1,\ldots,y_{d-1}))\cap I^d\subseteq (y_1,\ldots,y_{d-1})+I^{d+1}=(y_1,\ldots,y_{d-1})+y_dI^d.
% y_dx\in (I^{d+1}+(y_1,\ldots,y_{d-1}))\cap I^{d}=((y_1,\ldots,y_{d-1})\cap I^d)+I^{d+1}=(y_1,\ldots,y_{d-1})I^{d-1}+y_dI^{d}.
\end{equation*}
% % % $$ y_dx\in (I^{d+1}+(y_1,\ldots,y_{d-1}))\cap I^{d}=((y_1,\ldots,y_{d-1})\cap I^d)+I^{d+1}=(y_1,\ldots,y_{d-1})I^{d-1}+y_dI^{d}.$$
 Let $y_dx= \sum_{i=1}^{d-1}r_iy_i+sy_d$ where $s\in I^{d}$.
This implies $y_d(x-s)\in (y_1,\ldots, y_{d-1})$. So $x-s\in ((y_1,\ldots,y_{d-1}):y_d)\cap I^{d-1}$ and $\rho(\overline{x-s})=\alpha$. Hence $\rho$ is surjective.

Now let $x\in ((y_1,\ldots,y_{d-1}):y_d)\cap I^{d-1}$ such that $\rho(\bar{x})=\overline{\bar{x}t^{d-1}}=0$ in $[G(IR_{d-1})]_{d-1}$.
Then \begin{equation*}
x\in ((y_1,\ldots,y_{d-1}):y_d)\cap (I^d+(y_1,\ldots,y_{d-1}))=(y_1,\ldots,y_{d-1})+\left(((y_1,\ldots,y_{d-1}):y_d)\cap I^d\right).
     \end{equation*}
\textbf{Claim:} Let $n\geq d$ be an integer. Then \begin{equation*}
((y_1,\ldots,y_{d-1}):y_d)\cap I^n\subseteq (y_1,\ldots,y_{d-1})+((y_1,\ldots,y_{d-1}):y_d)\cap I^{n+1}).
                                    \end{equation*}
{\it{Proof of Claim.}} Let $x\in ((y_1,\ldots,y_{d-1}):y_d)\cap I^n$, then $y_dx\in (y_1,\ldots,y_{d-1})$. So
$$\bar{y}_dt~\cdot~\overline{\bar{x}t^n}=\overline{\overline{y_dx}t^{n+1}}=0 \text{~in~} [G(IR_{d-1})]_{n+1}$$
% % % So $a_d^*\bar{x}^*=\overline{a_dx}^*=0$
which implies
$\overline{\bar{x}t^n}\in [G(IR_{d-1})]_n$ is annihilated by some power of $\mathcal{M}$. Thus
$\overline{\bar{x}t^n}\in [\h_{\mathcal{M}}^0(G(IR_{d-1}))]_n=0$. Recall that $[\h_{\mathcal{M}}^0(G(IR_{d-1}))]_n=0$ for all $n\geq d$ by Lemma \ref{lemma-3}. This gives
that $x\in (y_1,\ldots,y_{d-1})+I^{n+1}$. So $x\in (y_1,\ldots,y_{d-1})+((y_1,\ldots,y_{d-1}):y_d)\cap I^{n+1})$.$\square$

By the above claim, $x\in (y_1,\ldots,y_{d-1})+((y_1,\ldots,y_{d-1}):y_d)\cap I^{n})\subseteq(y_1,\ldots,y_{d-1})+I^{n}$ for all $n\geq d$. This implies
$x\in (y_1,\ldots,y_{d-1})$ and $\rho$ is injective. Thus by \eqref{eqn-18/10}
% % $$[\h_{\mathcal{M}}^0(G(IR_{d-1}))]_{d-1}\cong \frac{(a_1,\ldots,a_{d-1}:a_d)\cap (I^{d-1}+(a_1,\ldots,a_{d-1}))}{(a_1,\ldots,a_{d-1})}.$$
% % \begin{equation}\label{eqn-main-formula-for-ed}
  $$e_d(I)=-\lambda_R\left(\frac{((y_1,\ldots,y_{d-1}):y_d)\cap (I^{d-1}+(y_1,\ldots,y_{d-1}))}{(y_1,\ldots,y_{d-1})}\right).$$
% %  \end{equation}
This completes the proof.
\end{proof}
\begin{theorem}\label{theorem-main-2}
  Let $(R,\m)$ be a Noetherian local ring of dimension $d\geq 2$ and $\depth R\geq d-1$. Let $Q$ be a parameter ideal of $R$ such that
  $b_i(G(Q))\geq 0$ for $0\leq i\leq d-2$.
% %   Suppose $\h_{\m}^{d-1}(R)$ is finitely generated. \sal{not needed?}
 Then
%  for $1\leq i\leq d$,
 \begin{equation*}
  -\lambda_R(\h_{\m}^{d-1}(R))\leq e_d(Q).
 \end{equation*}
\end{theorem}
\begin{proof}
 We may assume that $R/\m$ is infinite. Let $Q=(x_1,\ldots,x_d)$ such that $x_1^*,\ldots,x_d^*$ is a superficial sequence.
 For an integer $l\gg 0$, let $I=Q^l$ and $y_i=x_i^l$ for $1\leq i\leq d$. Then $y_1^*,\ldots,y_d^*$ is a superficial sequence with respect to $I$ and
\begin{equation}\label{equation:reduction-no-of-I}
 I^{d}=(y_1,\ldots,y_d)I^{d-1}.
\end{equation}
By Lemma \ref{lemma-4}, for $l\gg 0,$ $\depth G(I)\geq d-2.$
Since $b_{d-2}(G(Q))\geq 0$, we get $b_{d-2}(G(I))\geq \lfloor b_{d-2}(G(Q))/l \rfloor\geq 0$ by
\cite[Lemma 2.4]{Hoa}. Using \eqref{31-aug-6-e1}, we get that $\h_{\mathcal{M}}^{d-2}(G(I))=0$ for $l\gg 0$. Hence $\depth G(I)\geq d-1.$
By Lemma \ref{lemma-main-1},
\begin{align}
e_d(Q)=e_d(I)&=-\lambda_R\left(\frac{((y_1,\ldots,y_{d-1}):y_d)\cap (I^{d-1}+(y_1,\ldots,y_{d-1}))}{(y_1,\ldots,y_{d-1})}\right)\nonumber\\
 &\geq-\lambda_R\left(\frac{((y_1,\ldots,y_{d-1}):y_d)}{(y_1,\ldots,y_{d-1})}\right)\nonumber\\
 &\geq-\lambda_R(\h_{\m}^0(R/(y_1,\ldots,y_{d-1})))\label{eqn-ed-lower-bound}
\end{align}
where the last inequality holds since

  \begin{align*}
 \frac{((y_1,\ldots,y_{d-1}):y_d)}{(y_1,\ldots,y_{d-1})}\subseteq \h_{\m}^0(R/(y_1,\ldots,y_{d-1}))
 \end{align*}
Now let $R_i=R/(y_1,\ldots,y_{i})$ for $1\leq i\leq d-1$ and $R_0=R$. Note that $y_1,\ldots,y_{d-1}$ is a regular sequence in $R$ and $\depth R_i\geq d-i-1$.
% % and $x_{i+1}$ is regular on $R_i$.
For $0\leq i\leq d-2$, the exact sequence $$0\to R_i\xrightarrow{y_{i+1}} R_i\to R_{i+1}\to 0$$
gives the long exact sequence of local cohomology modules
$$0\to \h_{\m}^{d-i-2}(R_{i+1})\to \h_{\m}^{d-i-1}(R_{i})\xrightarrow{~y_{i+1}} \h_{\m}^{d-i-1}(R_{i})\to\ldots.$$
% % Since $x_{i+1}^l.\h_{\m}^{d-i-1}(R_{i})=0$ for $l\gg 0$, we get that
Thus for $0\leq i\leq d-2$,
\begin{equation}\label{eqn-length-of-H-m-decreasing}
 \lambda_R(\h_{\m}^{d-i-2}(R_{i+1}))\leq \lambda_R(\h_{\m}^{d-i-1}(R_{i})).
 \end{equation}
Putting the values of $i$ successively, we get
% \begin{equation*}
$$\lambda_R(\h_{\m}^0(R_{d-1}))\leq \lambda_R(\h_{\m}^1(R_{d-2}))\leq\ldots\leq \lambda_R(\h_{\m}^{d-1}(R)).$$
% \end{equation*}
Hence $e_d(Q)\geq -\lambda_R(\h_{\m}^{d-1}(R))$ by \eqref{eqn-ed-lower-bound}.
\end{proof}
\begin{corollary}\label{corr-ei-grteq}
  Let $(R,\m)$ be a Noetherian local ring of dimension $d\geq 2$. Let $Q$ be a parameter ideal of $R$
  such that $\depth G(Q)\geq d-1$.
% %   This implies $\depth R\geq d-1.$
 Then for $2\leq i\leq d$,
%  $-\lambda_R(\h_{\m}^{d-1}(R))\leq e_i(Q)\leq 0$ for $1\leq i\leq d$.
 \begin{equation}\label{eqn-23-3-corr}
   -\lambda_R(\h_{\m}^{d-1}(R))\leq e_i(Q).
  \end{equation}
\end{corollary}
\begin{proof}
We may assume that the residue field $R/\m$ is infinite. Let $Q=(y_1,\ldots,y_d)$ such that  $y_1^*,\ldots,y_d^*$ is a superficial sequence
in $G(Q)$.
Let $R_i=R/(y_1,\ldots,y_i)$ for $1\leq i\leq d-1$ and $R_0=R$.
% % Since $\depth R\geq d-1$, $x_1,\ldots,x_{d-1}$ is a regular sequence.
Since $y_1^*,\ldots,y_{d-1}^*$ is a regular sequence in $G(Q)$, $G(QR_{d-i})\cong G(Q)/(y_1^*,\ldots,y_{d-i}^*)G(Q)$ and 
$\depth G(QR_{d-i})\geq i-1.$
Hence
% by \sal{reference} and
by Theorem \ref{theorem-main-2},
\begin{equation}\label{eqn-bound-ei}-\lambda_R(\h_{\m}^{i-1}(R_{d-i}))\leq e_i(Q) \text{~~for ~~} 2\leq i\leq d.\end{equation}
Since $y_1,\ldots,y_{d-1}$ is a regular sequence in $R$, we have from \eqref{eqn-length-of-H-m-decreasing} that
$$\lambda_R(\h_{\m}^{i-1}(R_{d-i}))\leq \lambda_R(\h_{\m}^{i}(R_{d-i-1}))$$
for  $2\leq i\leq d-1$. This implies
% % For a fixed $1\leq i\leq d-1$, we put
% % $j=d-1-i,~d-1-i-1,\ldots,~0$  successively  into \eqref{eqn-length-of-H-m-decreasing} to obtain
$$\lambda_R(\h_{\m}^{i-1}(R_{d-i}))\leq \lambda_R(\h_{\m}^{i}(R_{d-i-1}))\leq \ldots \leq \lambda_R(\h_{\m}^{d-1}(R)).$$
Therefore by \eqref{eqn-bound-ei}, we get
\begin{equation*}
    -\lambda_R(\h_{\m}^{d-1}(R))\leq e_i(Q) \mbox{ for } 2\leq i\leq d.
\end{equation*}
\end{proof}
We include an example where \eqref{eqn-23-3-corr} does not hold.
\begin{example}
 We recall Example \ref{example-gcm-22-1} with $\dim R=d=5$ and $t=2$. Then $\dim A=5$ and $\depth A=2$.
 By \eqref{eqn-22-1}, $e_3(Q,A)<0$ for a parameter ideal $Q$ of $A$ where as $-\lambda_R(\h_{\m}^{d-1}(A))=0.$
\end{example}
\section{Vanishing of coefficients}\label{vanishing-of-coefficients}
In this section, we generalize partially a result of Goto and Ozeki \cite[Theorem 3.2]{goto-ozeki} on the vanishing of $e_2(Q)$ in two dimensional local rings.
In Theorem \ref{theorem-vanishing-of-ed}, we obtain a necessary condition similar to that given in \cite{goto-ozeki} for the vanishing of $e_d(Q)$ for a
parameter ideal $Q$ with certain conditions.
% % We recover \cite[Theorem 3.2]{goto-ozeki} as a corollary.
Further, we characterize the vanishing of
$e_2(Q)$ in rings of dimension $d\geq 2$ and $\depth $ at least $d-1$ in Theorem \ref{theorem-e2} which extends \cite[Theorem 3.2]{goto-ozeki} to
dimension $d$.

Let $x_0=0$. A sequence $x_1,\ldots,x_r\in\m$  is called a $d$-sequence if (i) $x_i$ is not in the ideal generated by the rest of the $x_j$ and (ii)
 $((x_0,\ldots,x_i): x_{i+1}x_j)=((x_0,\ldots,x_i): x_j)$ for $0\leq i\leq r-1$ and for all $j\geq i+1$. See \cite{huneke-d} and \cite{trung-abs-sup}
 for the theory of $d$-sequences. Note that the hypothesis of the following theorem is satisfied if $\depth G(Q)\geq d-1.$
\begin{theorem}\label{theorem-vanishing-of-ed}
 Let $(R,\m)$ be a Noetherian local ring of dimension $d\geq 2$ and $\depth R\geq d-1$. Let $Q=(x_1,\ldots,x_d)$ be a parameter ideal of $R$ such that
  $x_1^*,\ldots,x_{d-1}^*$ is a superficial sequence. Let $b_i(G(Q))\geq 0$ for $0\leq i\leq d-3$ and $[\h_\mathcal{M}^{d-2}(G(Q))]_n=0$ for all $n\ll 0$ i.e.
  $\h_\mathcal{M}^{d-2}G(Q)$ is finitely graded.
 Suppose $$e_d(Q)=0.$$ Then $x_1^l,\ldots,x_{d-1}^l,x_d^{(d-1)l}$ is a $d$-sequence in $R$ for all integers $l\geq 1$.
\end{theorem}
\begin{proof}
 For $l\gg 0,$ let $I=Q^l$ and $J=(x_1^l,\ldots,x_d^l).$ Then ${(x_1^l)}^*,\ldots,{(x_d^l)}^*$ is a superficial sequence with respect to $I$ and $I^d=JI^{d-1}.$
 Suppose $e_d(Q)=0$, then \eqref{equation:ed-negative} implies
$[\h_{\mathcal{M}}^{d-1}(G(I))]_0=0$. Hence by Lemma \ref{lemma-4},  $\depth G(I)\geq d-1$.

By Lemma \ref{lemma-main-1}, we get that
$$e_d(Q)=-\lambda_R\left(\frac{((x_1^l,\ldots,x_{d-1}^l):x_d^l)\cap (I^{d-1}+(x_1^l,\ldots,x_{d-1}^l))}{(x_1^l,\ldots,x_{d-1}^l)}\right).$$
 $e_d(Q)=0$ implies that
% % by \eqref{eqn-main-formula-for-ed} again,
% \begin{eqnarray*}
$((x^l_1,\ldots,x^l_{d-1}):x^l_d)\cap (I^{d-1}+(x^l_1,\ldots,x^l_{d-1}))=(x^l_1,\ldots,x^l_{d-1})$ for all $l\gg 0$.
% \end{eqnarray*}
Let $N\geq 1$ be an integer such that for all  $l\geq N$,
\begin{align}
((x^l_1,\ldots,x^l_{d-1}):x^l_d)\cap I^{d-1}\subseteq (x^l_1,\ldots,x^l_{d-1}).\label{eqn-claim-is-true-for-l>N}
\end{align}
% First we show that
{\bf{Claim:}} For all $l\geq 1$
\begin{equation}\label{eqn-containment-for-all-l}
((x^l_1,\ldots,x^l_{d-1}):x^l_d)\cap I^{d-1} \subseteq (x^l_1,\ldots,x^l_{d-1}).
% \text{ for all } l\geq 1.
\end{equation}
{\it {Proof of Claim.}}
Let $1\leq l< N$ and $y\in ((x^l_1,\ldots,x^l_{d-1}):x^l_d)\cap I^{d-1}$. Then
\begin{align*}
x_d^N\cdot x_1^{N-l}\cdot x_2^{N-l}\cdots x_{d-1}^{N-l}\cdot y&=x_d^{N-l}\cdot x_1^{N-l}\cdot x_2^{N-l}\cdots x_{d-1}^{N-l}\cdot x_d^ly\\
&\in x_d^{N-l}\cdot x_1^{N-l}\cdot x_2^{N-l}\cdots x_{d-1}^{N-l}(x_1^l,\ldots,x_{d-1}^l)\\
&\subseteq(x_1^N,\ldots,x_{d-1}^N).
\end{align*}
This implies \begin{align}
              x_1^{N-l}\cdot x_2^{N-l}\cdots x_{d-1}^{N-l}\cdot y&\in((x_1^N,\ldots,x_{d-1}^N):x_d^N)\cap Q^{(N-l)(d-1)}I^{d-1}\nonumber\\
              &=((x_1^N,\ldots,x_{d-1}^N):x_d^N)\cap Q^{N(d-1)}\nonumber\\
              &\subseteq(x^N_1,\ldots,x^N_{d-1})\label{eqn-d-seq}
      \end{align}
 where the last containment is due to \eqref{eqn-claim-is-true-for-l>N}.
 Now we show by induction on $d$ that if $y$ is such that $\eqref{eqn-d-seq}$ holds then $y\in (x^l_1,\ldots,x^l_{d-1})$.
 Note that $x_1,\ldots,x_{d-1}$ is a regular sequence. For $d=2$, $~x_1^{N-l}y\in (x_1^N)\implies y\in (x_1^l)$.
 Let $d>2$ and $\eqref{eqn-d-seq}$ holds. Set $y^{\prime}= x_1^{N-l}y$ and $R_1=R/(x_1^N)$. Let $\overline{\alpha}$ denote the image of an element
 $\alpha\in R$ in $R_1$. Then
\begin{align*}
 &x_2^{N-l}\cdot x_3^{N-l}\cdots x_{d-1}^{N-l}\cdot y^{\prime}
% % %  =(x_1^{N-l}.x_2^{N-l}\ldots x_{d-1}^{N-l})y
 \in (x^N_1,\ldots,x^N_{d-1})\\
 \implies&\overline{x}_2^{N-l}\cdot \overline{x}_3^{N-l}\cdots\overline{x}_{d-1}^{N-l}\cdot\overline{y^{\prime}}\in (\overline{x}_2^N,\ldots,\overline{x}^N_{d-1})R_1\\
 \implies& \overline{y^{\prime}}\in (\overline{x}_2^l,\ldots,\overline{x}^l_{d-1})R_1 \hspace{3cm} \text{ [by induction hypothesis]}\\
 \implies&   x_1^{N-l}y=y^{\prime}\in (x_2^l,\ldots,x_{d-1}^l)+(x_1^N)\\
 \implies& y\in (x_1^l,\ldots,x_{d-1}^l)
\end{align*}
where the last statement holds since $x_1^{N-l}$ is regular in $R/(x_2^l,\ldots,x_{d-1}^l)$. $\square$
% % \sal{quote this definition of dseq or cite }

To see that $x_1^l,\ldots,x_{d-1}^l,x_d^{(d-1)l}$ is a $d$-sequence in $R$, we use \eqref{eqn-containment-for-all-l} repeatedly. For this purpose, let $l\geq 1$ and
\begin{align*}
  r &\in ((x_1^l,\ldots,x_{d-1}^l):x_d^{(d-1)l})\cap (x_d^{(d-1)l})\\
 \implies rx_d^{(d-2)l} &\in  ((x_1^l,\ldots,x_{d-1}^l):x^l_d)\cap I^{d-1} \subseteq (x^l_1,\ldots,x^l_{d-1}) \\
  \implies rx_d^{(d-3)l}&\in  ((x_1^l,\ldots,x_{d-1}^l):x^l_d)\cap I^{d-1} \subseteq (x^l_1,\ldots,x^l_{d-1})\\
 &\vdots&\\
 \implies r&\in ((x_1^l,\ldots,x_{d-1}^l):x^l_d)\cap I^{d-1} \subseteq (x^l_1,\ldots,x^l_{d-1}).
\end{align*}
This implies that for $l\geq 1,$ $$((x_1^l,\ldots,x_{d-1}^l):x_d^{(d-1)l})\cap (x_1^l,\ldots,x_{d-1}^l,x_d^{(d-1)l})=(x^l_1,\ldots,x^l_{d-1}).$$
Since $x_1^l,\ldots,x_{d-1}^l$ is a regular sequence in $R$, it follows that
$x_1^l,\ldots,x_{d-1}^l,x_d^{(d-1)l}$ is a $d$-sequence for $l\geq 1$.
\end{proof}
In the next theorem, we obtain equivalent conditions for the equality $e_2(Q)=0$ which recovers \cite[Theorem 3.2]{goto-ozeki}.
It also unifies the necessary and sufficient conditions given by Mccune \cite[Theorem 3.5]{Lori} for $e_2(Q)=0.$
Recall that the postulation number of $Q$, denoted
 by $\eta(Q)$, is defined as
\begin{equation*}
\eta(Q):=\text{~min}\{i~|~H(Q,n)=P(Q,n) \text{  for all } n> i\}.
% % \label{postulation-number-for-Q-and-K}
\end{equation*}
\begin{theorem}\label{theorem-e2}
   Let $(R,\m)$ be a Noetherian local ring of dimension $d\geq 2$ and $\depth R \geq d-1$. Let $Q=(x_1,\ldots,x_d)$ be a parameter ideal of $R$ such that $x_1^*,\ldots,x_d^*$ is a superficial sequence.
   Then the following assertions hold.
   \begin{enumerate}
\item \label{theorem-e2-item1} $-\lambda_R(\h_{\m}^{d-1}(R))\leq e_2(Q)\leq 0$.
   \item \label{corr-e2-part-2} The following statements are equivalent:
   \begin{enumerate}[label=(\alph*)]
    \item \label{corr-e2-a} $e_2(Q)=0$;
    \item \label{corr-e2-b0} $x_1,\ldots,x_{d-2},x_{d-1}^l,x_d^l$ is $d$-sequence in $R$ for all integers $l\geq 1$;
    \item \label{corr-e2-b} $x_1,\ldots,x_d$ is $d$-sequence in $R$;
    \item \label{corr-e2-c} $\depth G(Q)\geq d-1$ and $\eta(Q)< 2-d$.
   \end{enumerate}
\item \label{theorem-e2-item3} $e_2(Q)=0\implies e_i(Q)=0$ for $2\leq i\leq d$.
   \end{enumerate}
\end{theorem}
\begin{proof}
\eqref{theorem-e2-item1}
Set $R_i=R/(x_1,\ldots,x_{i})$ for $1\leq i\leq d-1$ and $R_0=R$. Then $e_2(Q)=e_2(QR_{d-2})$ and by Theorems
% \sal{\cite[Proposition 1.2]{RV}}
\ref{theorem-main-(-1)} and \ref{theorem-main-2}, $-\lambda_R(\h_{\m}^{1}(R_{d-2}))\leq e_2(Q) \leq 0$. From \eqref{eqn-length-of-H-m-decreasing}, we have that
$$\lambda_R(\h_{\m}^1(R_{d-2}))\leq \lambda_R(\h_{\m}^2(R_{d-3}))\leq\ldots\leq \lambda_R(\h_{\m}^{d-1}(R)).$$
\eqref{corr-e2-part-2} \ref{corr-e2-a}$\implies$ \ref{corr-e2-b0}
% % $(a)\implies(b)$
% %  \eqref{corr-e2-a}$\implies$\eqref{corr-e2-b0},
 $e_2(Q)=0\implies e_2(QR_{d-2})=0$.
Therefore for all $l\geq 1$, the images of $x_{d-1}^l,x_d^l$ in $R_{d-2}$ is a $d$-sequence by 
Theorem \ref{theorem-vanishing-of-ed}. Since $x_1,\ldots,x_{d-1}$ is a
regular sequence in $R$, it follows that $x_1,\ldots,x_{d-2},x_{d-1}^l,x_{d}^l$
is a $d$-sequence in $R$.

\ref{corr-e2-b0}$\implies$\ref{corr-e2-b} It is obvious.

\ref{corr-e2-b}$\implies$\ref{corr-e2-c}
Since the images of $x_{d-1},x_d$
  in $R_{d-2}$ is a $d$-sequence, we have that $\h_{\mathcal{M}}^0(G(QR_{d-2}))\cong \h_{\m}^0(R_{d-2})$, see \cite[Proposition 3.4(5)]{goto-ozeki}. Since
  $\depth R_{d-2}\geq 1$, we get $\h_{\mathcal{M}}^0(G(QR_{d-2}))=0$ which implies  $\depth G(QR_{d-2})\geq 1$.
% %   \sal{write it} For the properties of $d$-sequence, see \cite[Proposition 3.4]{goto-ozeki}.
%   Thus by Proposition \ref{prop-pre-superficial-behave-G(Q)},
  Thus by Sally-machine,
  $\depth G(Q)\geq d-1$.
  Since the image of $x_d$ in $R_{d-1}$ is a $d$-sequence, by 
  \cite[Proposition 3.4(3)]{goto-ozeki}, $\eta(QR_{d-1})\leq 0$.
    Since $x_1^*,\ldots, x_{d-1}^*$ is a regular sequece in $G(Q)$ by \cite[Lemma 1.3]{RV} , we 
  get $\eta(Q)=\eta(QR_{d-1})-(d-1)\leq 1-d$ by \cite[Lemma 2.8]{Marley}.
%   \sal{reference}.

\ref{corr-e2-c}$\implies$\ref{corr-e2-a}
% %   for \eqref{corr-e2-c}$\implies$\eqref{corr-e2-a}, suppose
  $\eta(Q)\leq 1-d$, then $P(Q,n)=H(Q,n)=0$ for $n=0,-1,\ldots,2-d$. By putting the values of
  $n$ into $P(Q,n)$ successively, we easily get that $e_i(Q)=0$ for $2\leq i\leq d$.

\eqref{theorem-e2-item3} It follows from part \eqref{corr-e2-part-2}.
\end{proof}
The depth condition on the ring is necessary as evidenced by the following example.
\begin{example}
 In Example \ref{example-gcm-22-1}, let $\dim R=d\geq 4$ and $t=d-3$ so that $D=R/(X_1,X_2,X_3)$ and $A=R\ltimes D.$ Let $q=(X_1,\ldots,X_d)$
  and $Q=qA$. Then $e_2(Q,A)=0$ but  $\depth G(Q)=d-3$ and $e_3(Q,A)\neq 0$ by \eqref{eqn-22-1}. In this case, $\depth A=d-3.$
\end{example}
\nocite{*}

\noindent {\it Acknowledgement}: The second author would like to thank Indian Institute of Technology Guwahati for providing fellowship during the period when the research was carried out.

\addcontentsline{toc}{section}{References}
\bibliographystyle{plain}
% \bibliography{my_ref}
% \end{thebibliography}

\end{document}